\documentclass{amsproc}

\usepackage{amsmath}
\usepackage{amssymb}
\usepackage{amsthm}
\usepackage{mathrsfs}
\usepackage{amscd}

\newtheorem{theorem}{Theorem}[section]
\newtheorem{lemma}[theorem]{Lemma}

\theoremstyle{definition}
\newtheorem{definition}[theorem]{Definition}

\newtheorem{corollary}[theorem]{Corollary}
\newtheorem{proposition}[theorem]{Proposition}

\numberwithin{equation}{section}

\begin{document}

\title[$J$-Stability in non-Archimedean dynamics]{$J$-Stability of immediately expanding rational maps in non-Archimedean dynamical systems}


\author{LEE, Junghun}
\address{Graduate School of Mathematics, Nagoya University, 
Nagoya 464-8602, Japan}
\email{m12003v@math.nagoya-u.ac.jp}
\thanks{The author would like to thank to Professor Tomoki Kawahira.}


\subjclass[2010]{Primary 37P40, Secondly 11S82}

\date{November 15, 2014}

\begin{abstract}
The aim of this paper is to show $J$-stability of immediately expanding rational maps over an algebraically closed, complete, and non-Archimedean field, which is an analogue of R. Man\~e, P. Sad, and D. Sullivan's theorem of $J$-stability in complex dynamical systems.
\end{abstract}

\maketitle


\section{Introduction}

The theory of (discrete) dynamical systems investigates the iterations of a continuous map from a topological space to itself.
In the early twentieth century, P. Fatou and G. Julia developed the theory of complex dynamical systems. 
One of the most interesting topic in the theory of complex dynamical systems is the notion of the Fatou set and the Julia set because they are completely invariant under a given rational map so the dynamical systems on the Fatou set and the Julia set can be considered independently.
Moreover, the Fatou set and the Julia set are the stable locus and the chaotic locus of the complex dynamical systems, respectively.
See \cite{Miln06} for more details on the theory of complex dynamical systems.

In \cite{MSS83}, R. Man\~e, P. Sad, and D. Sullivan introduced the notion of $J$-stability of rational maps in the theory of complex dynamical systems.
Roughly speaking, $J$-stability means that the dynamical systems on the Julia sets of two given rational maps are dynamically equivalent if those two rational maps are close enough.
In the theory of complex dynamical systems, R. Man\~e, P. Sad, and D. Sullivan showed that a rational map is $J$-stable if it has a neighborhood in the set of rational maps on which the number of attracting cycles is constant.
See \cite{MSS83} and \cite{McMull94} for more details on $J$-stability in complex dynamical systems.

The purpose of this paper is to introduce an analogue of R. Man\~e, P. Sad, and D. Sullivan's theorem in {\it non-Archimedean dynamical systems}, which is a dynamics on the projective line over algebraically closed, complete, and non-Archimedean fields generated by a rational map over the same field.

In the theory of non-Archimedean dynamical systems, the {\it Fatou set} is defined as the largest open set on which the iterations of a rational map is equicontinuous and the {\it Julia set} is defined as the complement of the Fatou set with respect to the projective line.
Moreover, the Fatou set and the Julia set are completely invariant under a given rational map so the dynamical systems on the Fatou set and the Julia set can be considered independently as in the theory of complex dynamical systems.
We will briefly review a number of fundamental knowledge on the theory of non-Archimedean dynamical systems in Section $2$.

The theory of non-Archimedean dynamical systems was mainly developed by L-C. Hsia, R. Benedetto and J. Rivera-Letelier in the early twenty-first century.
For instance, an analogue of Montel's theorem in the theory of non-Archimedean dynamical systems, the existence of a polynomial map having a wandering domain and an analogue of the classification theorem of the Fatou component of the theory of complex dynamical systems in the theory of non-Archimedean dynamical systems were proved by L-C. Hsia in \cite{Hs00}, R. Benedetto in \cite{Ben02}, and J. Rivera-Letelier in \cite{Riv03}, respectively.
On the other hand, J. Kiwi and L-C. Hsia and  J-Y. Briend showed that the Julia sets of some polynomial maps are topologically conjugate to the symbolic dynamics in \cite{Kiw06} and \cite{BH12}, respectively.
In particular, J. Kiwi showed in \cite{Kiw06} that the dynamical systems of cubic polynomial maps on the completion of the field of formal Puiseux series is closely related to the structure of the parameter space of complex cubic polynomial near infinity.

The main theorem of this paper will be stated in the end of this section.
Let us begin with the definition of immediately expanding rational maps.

\begin{definition}\label{3.1}
Let $K$ be an algebraically closed field with a complete and non-Archimedean norm $|\cdot|$ and $f : \mathbb{P}^1_{K} \rightarrow \mathbb{P}^1_{K}$ be a non-constant rational map over $K$.
We say that the rational map $f$ is {\it immediately expanding} if it satisfies the following properties:
\begin{enumerate}
\item The Julia set $\mathcal{J}_f$ of $f$ is non-empty.
\item There exists a $\lambda > 1$ such that for any $z$ in $\mathcal{J}_f$, we have
$$
|f'(z)| \geq \lambda.
$$
\end{enumerate}
\end{definition}

Now we state the main theorem of this paper.

\begin{theorem}\label{3.2}
Let $K$ be an algebraically closed field of characteristic zero with a complete and non-Archimedean norm and $f : \mathbb{P}^1_{K} \rightarrow \mathbb{P}^1_{K}$ be a rational map over $K$ of degree $d$ which is greater than $2$.
Suppose that $f$ is immediately expanding and $\infty$ is not in the Julia set $\mathcal{J}_f$ of $f$.
Then there exists a neighborhood $U$ of $f$ in the set $\mathcal{R}_d$ of rational maps of degree $d$ such that for any $g$ in $U$, there exists a homeomorphism $h : \mathcal{J}_f \rightarrow \mathcal{J}_g$ such that for any $z$ in $\mathcal{J}_f$, we have
$$
h \circ f (z) = g \circ h (z).
$$
\end{theorem}

We say such a rational map $f$ is {\it $J$-stable} in the set $\mathcal{R}_d$ of rational maps of degree $d$.
In the statement of Theorem \ref{3.2}, the topology on the set of rational maps of degree $d$ is induced from the $2d - 1$ dimensional projective space over the same field.

We remark that Theorem \ref{3.2} holds for any topology on the set of rational maps of degree $d$ if the topology satisfies a certain property, see Lemma \ref{4.7}.
We also remark that the assumption that the infinity is not contained in the Fatou set is essential because of the existence of the Fatou set, see Proposition \ref{2.6}.

\subsection*{Contents of this paper}

In Section $2$, we will review fundamental knowledge such as the definitions of the projective line and the chordal metric and some properties of rational maps of non-Archimedean dynamical systems.

In Section $3$, we will give the definitions of the set of rational maps preserving a given subset and the sequence of subsets generated by a given rational map.
We will also see their properties as lemmas for the main theorem.

In Section $4$, we will prove the main theorem, Theorem \ref{3.2}.


\section{Preliminaries}

In this section, we will review some basic notions of non-Archimedean dynamical systems such as the definitions of projective line, the chordal metric, and the Julia set and some properties of rational maps.

Let $K$ be an algebraically closed field of characteristic zero with a complete, multiplicative, and non-Archimedean norm $|\cdot|$.
For instance, one can consider $K$ as the field of $p$-adic complex numbers with its $p$-adic norm or the completion of the field of formal Puiseux series with its natural norm, see \cite{Rob00} or \cite[Section $2.1$]{Kiw06}.
Let us first see the definitions of the projective line over $K$ and its chordal metric.

\begin{definition}\label{2.1}
The {\it projective line} $\mathbb{P}^1_{K}$ over $K$ is defined as the union of $K$ and the set of infinity $\infty$, that is, 
$$
\mathbb{P}^1_K := K \cup \{ \infty \}.
$$

Moreover, the {\it chordal metric} $\rho$ on $\mathbb{P}^1_{K}$ is defined as follows: For any $z$ and $w$ in $\mathbb{P}^1_{K}$, 
\begin{align*}
\rho(z, w) :=
\begin{cases}
\displaystyle {|z - w| \over \max\{1, |z| \} \cdot \max\{1, |w| \}} \quad &(z, w \in K),\\\\
\displaystyle {1 \over \max\{1, |z| \}} \quad &(z \in K, w = \infty).
\end{cases}
\end{align*}
\end{definition}

Next let us recall the definition of rational maps.

\begin{definition}\label{2.2}
Let $d$ be a natural number.
Then we say that $f$ is a {\it rational map} over $K$ of degree $d$ if there exist two variables homogeneous polynomial maps $f_1$ and $f_2$ over $K$ of degree $d$ such that $f_1$ and $f_2$ has no common factors and $f$ is determined by $f_1$ and $f_2$.
\end{definition}

Thus, we can consider a rational map over $K$ of degree $d$ as the quotient of two polynomial maps, of which the maximal degree is $d$, over $K$ with no common factors.
We shall use $\mathcal{R}_d$ to denote the set of rational maps over $K$ of degree $d$.
One can easily check that the set of rational maps over $K$ of degree $d$ can be naturally identified with an open subset of $2d + 1$ dimensional projective space over $K$, see \cite[Proposition $2.13$ or Section $4.3$]{Silv07} for more details.

In the rest of this section, we will mainly consider properties of rational maps.
We shall use $B_{a, r}$ to denote the {\it closed ball} in $K$ centered at $a$ with radius $r$, that is, 
$$
B_{a, r} := \{ z \in K \mid |z - a | \leq r \}.
$$ 

Let us begin with topological properties such as continuity and openness of rational maps.

\begin{proposition}\label{2.3}
Let $f : \mathbb{P}^1_K \rightarrow \mathbb{P}^1_K$ be a non-constant rational map over $K$.
Then we have the followings:
\begin{enumerate}
\item The rational map $f$ is continuous and open with respect to the chordal metric $\rho$.
\item If $f$ has no poles and zeros in $B_{a, r}$, then we have
$$
|f(z)| = |f(a)|
$$
for any $z$ in $B_{a, r}$.
\end{enumerate}
\end{proposition}

\begin{proof}[Proof of Proposition \ref{2.3}]
See \cite[Corollary $5.17$ and Theorem $5.13$]{Silv07} for the proofs.
\end{proof}

Next, we recall the definitions of the Fatou set and the Julia set.
We shall use $f^k$ to denote the {\it $k$-th iteration} of a given map $f : \mathbb{P}^1_K \rightarrow \mathbb{P}^1_K$, that is, 
$$
f^k := \underbrace{f \circ f \circ \cdots \circ f}_{k \text{ times}}.
$$
\begin{definition}\label{2.4}
Let $f : \mathbb{P}^1_{K} \rightarrow \mathbb{P}^1_{K}$ be a non-constant rational map over $K$.
\begin{enumerate}
\item The {\it Fatou set} $\mathcal{F}_f$ of $f$ is defined as the largest open set in $\mathbb{P}^1_{K}$ on which the family $\{ f^k \}_{k = 0}^{\infty}$ of the iterations of $f$ is equicontinous with respect to the chordal metric $\rho$.
\item The {\it Julia set} $\mathcal{J}_f$ of $f$ is defined as the complement of the Fatou set with respect to the projective line $\mathbb{P}^1_K$.
\end{enumerate}
\end{definition}

The Fatou set and the Julia set have properties as follows:

\begin{proposition}\label{2.6}
Let $f : \mathbb{P}^1_{K} \rightarrow \mathbb{P}^1_{K}$ be a non-constant rational map over $K$.
Then we have the followings:
\begin{enumerate}
\item The Fatou set $\mathcal{F}_f$ of $f$ is non-empty.
\item The Julia set $\mathcal{J}_f$ is completely invariant, that is, 
$$
f(\mathcal{J}_f) = \mathcal{J}_f \quad \text{and} \quad f^{-1}(\mathcal{J}_f) = \mathcal{J}_f.
$$
\item The Julia set $\mathcal{J}_f$ of $f$ is equal to the Julia set $J_{f^k}$ of $f^k$ for any $k$ in $\mathbb{N}$.
\item The Julia set $\mathcal{J}_f$ is equal to the topological closure of the backward orbit of any point in $\mathcal{J}_f$, that is, for any $z$ in $\mathcal{J}_f$, we have
$$
\mathcal{J}_f = \overline{\bigcup_{k = 0}^{\infty} f^{-k}(\{ z \})}.
$$
\end{enumerate}
\end{proposition}

\begin{proof}[Proof of Proposition \ref{2.6}]
See \cite[Corollary $5.19$]{Silv07} and \cite[Proposition $5.18$ and Corollary $5.32$]{Silv07} for the proofs.
\end{proof}

It is well-known that the Julia set is closely related to the set of repelling periodic points.
Moreover, it will be useful in the proof of the main theorem.
Let us first recall the definitions of periodic points and repelling periodic points.

\begin{definition}\label{2.7}
Let $f : \mathbb{P}^1_K \rightarrow \mathbb{P}^1_K$ be a non-constant rational map over $K$.
\begin{enumerate}
\item We say that a point $p$ in $\mathbb{P}^1_K$ is a {\it periodic point} of $f$ with period $q$ if
$$
f^{q}(p) = p.
$$
\item We say that a periodic point $p$ in $K$ with period $q$ is repelling if 
$$
|(f^q)'(p)| > 1.
$$
\end{enumerate}
\end{definition}

\begin{lemma}\label{2.8}
Let $f : \mathbb{P}^1_K \rightarrow \mathbb{P}^1_K$ be a non-constant rational map over $K$.
Then we have the followings:
\begin{enumerate}
\item Every repelling periodic point is contained in the Julia set $\mathcal{J}_f$ of $f$.
\item The Julia set $\mathcal{J}$ of $f$ is contained in the topological closure of the set of periodic points of $f$. 
\end{enumerate}
\end{lemma}

\begin{proof}[Proof of Lemma \ref{2.8}]
See \cite[Proposition $5.20$]{Silv07} and \cite[Theorem 3.1]{Hs00} or \cite[Theorem 5.37]{Silv07} for the proofs.
\end{proof}

It is important to find zeros of a given rational map because a backward invariant subset of a given rational map will be considered in Section $3$.
The existence and locus of zeros can be calculated by the following lemma.

\begin{lemma}\label{2.9}
Let $f : \mathbb{P}^1_{K} \rightarrow \mathbb{P}^1_{K}$ be a non-constant rational map over $K$, $a$ be an element in $K$ with $f(a) \neq \infty$, and $r$ be a positive real number.
Suppose that for any $z$ in $B_{a, r}$, we have
$$
f(z) = a_0 + a_1 \cdot (z - a) + a_2 \cdot (z - a)^2 + \cdots
$$ 
where $\{ a_k \}_{k = 0}^{\infty}$ is a sequence in $K$.
Setting a natural number $l$ by 
$$
\max\{ m \in \{ 0, 1, \cdots \} \mid \forall n \in \{0, 1, \cdots \},  | a_{m} | \cdot r^{m} \geq | a_{n} | \cdot r^{n} \},
$$
we have exactly $l$ elements $\{ \alpha_n \}_{n = 1}^{l}$ in $B_{a, r}$, counted with multiplicity, such that for any $n$ in $\{ 1, 2, \cdots, l \}$, we have
$$
f(\alpha_n) = 0.
$$
In particular, if the natural number $l$ is equal to $0$, the rational map $f$ has no zeros in $B_{a, r}$.
\end{lemma}

\begin{proof}[Proof of Lemma \ref{2.9}]
See \cite[Theorem $5.11$]{Silv07} for the proof.
\end{proof}

As an immediate consequence of Lemma \ref{2.9}, the following corollary, which implies that the image of a pole-free closed ball under a given rational map is also a closed ball, can be obtained.

\begin{corollary}\label{2.10}
Let $f : \mathbb{P}^1_{K} \rightarrow \mathbb{P}^1_{K}$ be a non-constant rational map over $K$, $a$ be an element in $K$ with $f(a) \neq \infty$, and $r$ be a positive real number.
Suppose that for any $z$ in $B_{a, r}$, we have
$$
f(z) = a_0 + a_1 \cdot (z - a) + a_2 \cdot (z - a)^2 + \cdots
$$ 
where $\{ a_k \}_{k = 0}^{\infty}$ is a sequence in $K$.
Setting a natural number $l$ by 
$$
\max\{ m \in \mathbb{N} \mid \forall n \in \mathbb{N},  | a_{m} | \cdot r^{m} \geq | a_{n} | \cdot r^{n} \},
$$
we have 
$$
f(B_{a, r}) = B_{f(a), |a_l| \cdot r^{l}}.
$$
\end{corollary}

\begin{proof}[Proof of Corollary \ref{2.10}]
See \cite[Proposition $5.16$]{Silv07} for the proof.
\end{proof}

Let us close this section by proving the existence of one of the most important constant in this paper.
We remark that the constant will appear quite frequently in this paper.

\begin{proposition}\label{2.11}
The following non-zero constant exists:
$$
\min \{ |k|^{1 / (k-1)} \mid k \in \{2, 3, \cdots \} \}.
$$
\end{proposition}

\begin{proof}[Proof of Proposition \ref{2.11}]
Assume that the constant
$$
p : = \min \{ l \in \mathbb{N} \mid |l| < 1 \}
$$
exists.
Then, it follows from the multiplicativity of the norm $|\cdot|$ and the minimality of $p$ that $p$ must be a prime number. 
Moreover, if a natural number $m$ is co-prime to $p$, then we have
$$
|m| = 1.
$$
Hence, we have 
\begin{align*}
\min \{ |k|^{1 / (k - 1)} \mid k \in \{2, 3, \cdots \} \} 
&= \min \{ |p|^{n / (p^n - 1)} \mid n \in \mathbb{N} \} \\
&= |p|^{1 / (p-1)}.
\end{align*}
\end{proof}

We remark that the constant is always zero if the characteristic of $K$ is non-zero.


\section{Key Lemmas}

In this section, we will introduce notation such as the set of rational maps preserving a given subset and the sequence of subsets generated by a given subset and a given rational map, which will be used in Section $4$.

Let us begin with the definition of the set of rational maps preserving a subset.

\begin{definition}\label{4.1}
Let $\Omega$ be a subset of $K$, $\lambda$ be a real number which is greater than $1$, and $d$ be a natural number.
Then the {\it set $\mathcal{N}_{\lambda, \Omega}$ of rational maps of degree $d$ preserving $\Omega$} is defined by 
$$
\{ g \in \mathcal{R}_d \mid \forall z \in \Omega, |g'(z)| \geq \lambda, g^{-1}(\{ z \}) \subset \Omega \}
$$
where $\mathcal{R}_d$ is the set of rational maps over $K$ of degree $d$.
\end{definition}

The following lemma is essential for the construction of a topological conjugacy between the Julia sets of two rational maps when those two rational maps are close enough.

\begin{lemma}\label{4.2}
Let $\mathcal{N}_{\lambda, \Omega}$ be the set of rational maps of degree $d$ preserving a subset $\Omega$ and $f$ be an element in $\mathcal{N}_{\lambda, \Omega}$.
Let us assume the conditions as follows:
\begin{enumerate}
\item The rational map $f$ has no critical values and poles in $\Omega$.
\item There exists a $\delta > 0$ such that for any $z$ in $\Omega$, we have
$$
B_{z, \delta} \subset \Omega.
$$
\end{enumerate}
Then, for any $z$ in $\Omega$, any $k$ in $\{1, 2, \cdots, d \}$, and any non-negative real number $r$ which is less than or equal to $\mu$, the restriction map 
$$
f| B_{z_k, r / |f'(z_k)|} \rightarrow B_{z, r}
$$ 
is homeomorphic and we have 
$$
f^{-1}(B_{z, r}) = \bigsqcup_{k=1}^{d} B_{z_k, r / |f'(z_k)|}
$$
where 
$$
f^{-1}(z) = \{ z_1, z_2, \cdots, z_d\}
$$ 
and 
$$
\mu := \min \{ |l|^{1 / (l - 1)} \mid l \in \{ 2, 3, \cdots \} \} \cdot \delta > 0.
$$
\end{lemma}

Note that the existence of $\mu$ follows from Proposition \ref{2.11}.

\begin{proof}[Proof of Lemma \ref{4.2}]
Since $f$ has no poles in $B_{z_k, \delta}$, there exists a sequence $\{ a_m \}_{m = 0}^{\infty}$ in $K$ such that for any $w$ in $B_{z_k, \delta}$, we have
$$
f(w) = f(z_k) + a_1 \cdot (w - z_k) + a_2 \cdot (w - z_k)^2 + \cdots
$$
and
$$
f'(w) = a_1 + 2 \cdot a_2 \cdot (w - z_k) + 3 \cdot a_3 \cdot (w - z_k)^2 + \cdots.
$$
Moreover, since $f'$ has no zeros in $B_{z_k, \delta}$, we have 
$$
|a_1| > |n| \cdot |a_n| \cdot \delta^{n -1}
$$
for any $n$ in $\{ 2, 3, \cdots \}$ by Lemma \ref{2.9}.
Thus, for any $n$ in $\{ 2, 3, \cdots \}$, we have
\begin{align*}
|a_1| \cdot {\mu \over |f'(z_k)|} 
&> |n| \cdot |a_n | \cdot \delta^{n -1} \cdot {\mu \over |f'(z_k)|} \\
&\geq |a_n| \cdot \mu^{n - 1} \cdot {\mu \over |f'(z_k)|}
= |a_n| \cdot \left( {\mu \over |f'(z_k)|} \right)^n.
\end{align*}
This implies that 
\begin{align*}
|a_1| \cdot {r \over |f'(z_k)|}
&= |a_1| \cdot {\mu \over |f'(z_k)|} \cdot {r \over \mu} \\
&> |a_n| \cdot \left( {\mu \over |f'(z_k)|} \right)^n \cdot {r \over \mu} \\
&\geq |a_n| \cdot \left( {\mu \over |f'(z_k)|} \right)^n \cdot \left( {r \over \mu} \right)^{n}
= |a_n| \cdot \left( {r \over |f'(z_k)|} \right)^n
\end{align*}
for any $n$ in $\{2, 3, \cdots \}$ and any non-negative real number $r$ which is less than or equal to $\mu$.
It follows from Proposition \ref{2.3}, Lemma \ref{2.9}, and Corollary \ref{2.10} that for any non-negative real number $r$ which is less than or equal to $\mu$, the restriction map 
$$
f | B_{z_k, r / |f'(z_k)|} \rightarrow B_{z, r}
$$
is homeomorphic.
On the other hand, it is easy to see that for any $l \neq k$ in $\{1, 2, \cdots, d \}$, we have
$$
B_{z_k, r/ |f'(z_k)|} \cap B_{z_l, r/ |f'(z_l)|} = \emptyset.
$$
\end{proof}

Before constructing the topological conjugacy between the Julia set of two close enough rational maps, we first introduce a decreasing sequence of subsets.

\begin{definition}\label{4.3}
Let $\mathcal{N}_{\lambda, \Omega}$ be the set of rational maps of degree $d$ preserving a subset $\Omega$ and $f$ be an element in $\mathcal{N}_{\lambda, \Omega}$.
Then the {\it sequence $\{ \Omega_{k, f} \}_{k = 0}^{\infty}$ of subsets generated by $\Omega$ and $f$} is defined as follows:
\begin{enumerate}
\item The {\it $k$-th term} $\Omega_{k, f}$ is defined by 
$$
f^{-k}(\Omega).
$$  
\item The {\it limit term} $\Omega_{\infty, f}$ is defined by 
$$
\bigcap_{k = 0}^{\infty} \Omega_{k, f}.
$$
\end{enumerate}
\end{definition}

The following lemma shows basic properties of the sequence.

\begin{lemma}\label{4.4}
Let $\mathcal{N}_{\lambda, \Omega}$ be the set of rational maps of degree $d$ preserving a subset $\Omega$, $f$ be an element in $\mathcal{N}_{\lambda, \Omega}$, and $\{ \Omega_{k, f} \}_{k = 0}^{\infty}$ be the sequence of subsets generated by $\Omega$ and $f$.
Then we have
$$
\Omega_{\infty, f} \subset \cdots \subset \Omega_{k + 1, f} \subset \Omega_{k, f} \subset \cdots \subset \Omega_{1, f} \subset \Omega_{0, f}.
$$
Moreover, the limit term $\Omega_{\infty, f}$ is completely invariant under $f$.
\end{lemma}

\begin{proof}[Proof of Lemma \ref{4.4}]
It is trivial from Definition \ref{4.3}.
\end{proof}

Now we prepare a lemma for the construction of a topological conjugacy.

\begin{lemma}\label{4.5}
Let $\mathcal{N}_{\lambda, \Omega}$ be the set of rational maps of degree $d$ preserving a subset $\Omega$ and $f$ and $g$ be elements in $\mathcal{N}_{\lambda, \Omega}$.
Let us assume the conditions as follows:
\begin{enumerate}
\item The rational map $f$ has no critical values and poles in $\Omega$.
\item There exists a $\delta > 0$ such that for any $z$ in $\Omega$, we have
$$
B_{z, \delta} \subset \Omega.
$$
\item For any $z$ in $\Omega$, we have
$$
|f(z) - g(z)| \leq \mu
$$
where
$$
\mu := \min \{ |k|^{1 / (k - 1)} \mid k \in \{ 2, 3, \cdots \} \} \cdot \delta > 0.
$$
\end{enumerate}
Then, there exists a sequence $\{ h_l : \Omega_{l, f} \rightarrow \Omega_{l, g} \}_{l = 0}^{\infty}$ of homeomorphisms such that for any $l$ in $\{0, 1, \cdots \}$, any $z$ in $\Omega_{l + 1, f}$, and any $w$ in $\Omega_{l + 1, g}$, we have
$$
h_{l} \circ f (z) = g \circ h_{l+1} (z),
$$
$$
|h_{l + 1}(z) - h_l(z)| \leq {\mu \over \lambda^l},
$$
and
$$
|h^{-1}_{l + 1}(w) - h^{-1}_l(w)| \leq {\mu \over \lambda^l}
$$
where $\{ \Omega_{l, f} \}_{l = 0}^{\infty}$ is the sequence of subsets generated by $\Omega$ and $f$ and $\{ \Omega_{l, g} \}_{l = 0}^{\infty}$ is the sequence of subsets generated by $\Omega$ and $g$. 
\end{lemma}

\begin{proof}[Proof of Lemma \ref{4.5}]
Setting
\begin{align*}
h_0 : \Omega_{0, f} &\rightarrow \Omega_{0, g} \\
z &\mapsto z,
\end{align*}
we have a homeomorphism between $\Omega_{0, f}$ and $\Omega_{0, g}$ since $\Omega_{0, f}$ is equal to $\Omega_{0, g}$. 
Next we construct a continuous and open map $h_1 : \Omega_{1, f} \rightarrow \Omega_{1, g}$ as follows: For any $z$ in $\Omega_{1, f}$, there exists a unique $w$ in $K$ such that
$$
w \in  g^{-1}(\{ h_0 \circ f(z) \})
$$
and
$$
w \in B_{z, \mu / |g'(z)|}.
$$
Thus we can set
\begin{align*}
h_1 : \Omega_{1, f} &\rightarrow \Omega_{1, g} \\
z &\mapsto w.
\end{align*}
Then for any $z$ in $\Omega_{1, f}$, we have
$$
h_0 \circ f (z) = g \circ h_1 (z)
$$
and
$$
| h_{1} (z) - h_0(z) | \leq { \mu \over \lambda}
$$
and the map $h_1 : \Omega_{1, f} \rightarrow \Omega_{1, g}$ is continuous and open by Lemma \ref{4.2}.
Repeating this process, we can construct a sequence $\{ h_l : \Omega_{l, f} \rightarrow \Omega_{l, g} \}_{l = 0}^{\infty}$ of continuous and open maps such that for any $l$ in $\{ 0, 1, \cdots \}$ and any $z$ in $\Omega_{l + 1, f}$, we have
$$
h_l  \circ f (z) = g \circ h_{l + 1} (z)
$$
and
$$
| h_{l + 1} (z) - h_{l}(z) | \leq { \mu \over \lambda^{l}}.
$$
Similarly, we can also construct a sequence $\{ \tilde{h}_l : \Omega_{l, g} \rightarrow \Omega_{l, f} \}_{l = 0}^{\infty}$ of continuous and open maps such that for any $l$ in $\{ 0, 1, \cdots \}$ and any $w$ in $\Omega_{l + 1, g}$, we have
$$
f \circ \tilde{h}_{l + 1} (w) = \tilde{h}_{l} \circ g (w)
$$
and
$$
| \tilde{h}_{l + 1} (w) - \tilde{h}_{l}(w) | \leq { \mu \over \lambda^{l}}.
$$

Moreover, it is clear that for any $l$ in $\{ 0 , 1, \cdots \}$, any $z$ in $\Omega_{l, f}$, and any $w$ in $\Omega_{l, g}$, we have
$$
h_l \circ \tilde{h}_l (w) = w
$$
and
$$
\tilde{h}_l \circ h_l (z) = z.
$$
This implies that the map $h_l : \Omega_{l, f} \rightarrow \Omega_{l, g}$ is homeomorphic and $\tilde{h}_l : \Omega_{l, g} \rightarrow \Omega_{l, f}$ is the inverse map for any $l$ in $\{0, 1, \cdots \}$.
\end{proof}

\begin{lemma}\label{4.6}
Let $\mathcal{N}_{\lambda, \Omega}$ be the set of rational maps of degree $d$ preserving a closed subset $\Omega$ and $f$ and $g$ be elements in $\mathcal{N}_{\lambda, \Omega}$.
Suppose that there exists a sequence $\{ h_k : \Omega_{k, f} \rightarrow \Omega_{k, g} \}_{k = 0}^{\infty}$ of homeomorphisms such that for any $k$ in $\mathbb{N}$ and any $z$ in $\Omega_{k + 1, f}$, we have
$$
h_{k} \circ f (z) = g \circ h_{k + 1} (z),
$$
$$
|h_{k + 1}(z) - h_k(z)| \leq {\mu \over \lambda^k},
$$
and
$$
|h^{-1}_{k + 1}(z) - h^{-1}_k(z)| \leq {\mu \over \lambda^k}
$$
for some $\mu > 0$.
Then, there exists a homeomorphism $h_{\infty} : \Omega_{\infty, f} \rightarrow \Omega_{\infty, g}$ such that for any $z$ in $\Omega_{\infty, f}$, we have
$$
h_\infty \circ f (z) = g \circ h_\infty (z).
$$
\end{lemma}

\begin{proof}[Proof of Lemma \ref{4.6}]
Since $\Omega_{\infty, f}$ is complete, a topological space 
$$
\{ \phi : \Omega_{\infty, f} \rightarrow \Omega_{\infty, g} \mid \phi : \text{ continuous with respect to $\rho$} \}
$$
is also complete with respect to a metric
$$
\sup \{\rho(\phi_1 (z), \phi_2 (z)) \mid z \in \Omega_{\infty, f} \}.
$$ 
Moreover, it is clear that the sequence $\{ h_k : \Omega_{k, f} \rightarrow \Omega_{k, g} \}_{k = 0}^{\infty}$ of homeomorphisms is Cauchy so there exists a homeomorphism $h_{\infty} : \Omega_{\infty, f} \rightarrow \Omega_{\infty, g}$ such that 
$$
\lim_{k \rightarrow \infty} \sup \{ \rho( h_{\infty}(z), h_k(z)) \mid z \in \Omega_{\infty, f} \} = 0.
$$
Furthermore, since any rational map over $K$ is continuous, we have
\begin{align*}
h_{\infty} \circ f(z) 
&= \lim_{m \rightarrow \infty}h_{m} \circ f(z) \\
&= \lim_{m \rightarrow \infty}g \circ h_{m + 1}(z) \\
&= g \circ \lim_{m \rightarrow \infty} h_{m + 1}(z)
= g \circ h_{\infty} (z).
\end{align*}
for any $z$ in $\Omega_{\infty, f}$.
\end{proof}

Next let us consider a few basic properties of the topology of the set of rational maps.

\begin{lemma}\label {4.7}
Let $\Omega$ be a closed subset in $K$ and $f$ be an element in the set $\mathcal{R}_d$ of rational maps of degree $d$.
Let us assume the conditions as follows:
\begin{enumerate}
\item The rational map $f$ has no critical values and poles in $\Omega$.
\item There exists a $\delta > 0$ such that for any $z$ in $\Omega$, we have
$$
B_{z, \delta} \subset \Omega.
$$
\item There exists an $\eta > 0$ such that
$$
\Omega \subset B_{0, \eta}.
$$
\end{enumerate}
Then for any $r > 0$ and $s > 0$, there exists a neighborhood $U$ of $f$ in $\mathcal{R}_d$ such that for any rational map $g$ in $U$ and any $z$ in $\Omega$, we have
$$
|f(z) - g(z) | < r
$$
and
$$
|f'(z) - g'(z)| < s.
$$
In particular, any rational map $g$ in $U$ has no poles in $\Omega$.
\end{lemma}

\begin{proof}[Proof of Lemma \ref{4.7}]
The rational map $f$ can be written as 
$$
f(z) = {f_1(z) \over f_2(z)}
$$
where $f_1$ and $f_2$ are polynomial maps over $K$ with no common factors.
We shall denote the maximal value of a polynomial map $f_l$ and $f'_l$ in $B_{0, \eta}$ by $M_l$ and $M'_l$ for each $l$ in $\{ 1, 2 \}$, respectively. That is, for each $l$ in $\{ 1, 2 \}$, we set
$$
M_l := \max \{ |f_l(z)| \mid z \in B_{0, \eta} \}
$$
and
$$
M'_l := \max \{ |f'_l(z)| \mid z \in B_{0, \eta} \} .
$$
It is easy to check that these constants $M_l$ and $M'_l$ exists for any $l$ in $\{ 1, 2 \}$ even if $B_{0, \eta}$ is non-compact, see \cite[Theorem $5.13$]{Silv07} for more details.
Similarly, we shall also denote the minimal value of a polynomial map $f_2(z)$ in $\Omega$ by $m_2$, that is,
$$
m_2 := \min \{ |f_2(z)| \mid z \in \Omega \}.
$$
It is also easy to check that this constant $m_2$ is also exists since $f_2$ has no zero in $\Omega$.

For each $i$ and $j$ in $\{0, 1, \cdots, d \}$, we first set the constants
$$
\xi_i := {r \over M_{2} \cdot \eta^{i} } > 0
$$
and
$$
\zeta_j := \min \{ { m_{2} \over \eta^j}, {r \over M_{2}^2 \cdot \eta^j } \cdot M_{1} \} > 0
$$
Then for any rational map $g = g_1 / g_2$ in $\mathcal{R}_d$ with 
$$
g_1(z) = f_1(z) + \sum_{i = 0}^{d} \epsilon_i \cdot z^i
$$
and
$$
g_2(z) = f_2(z) + \sum_{j = 0}^{d} \kappa_j \cdot z^j
$$ 
where 
$$
|\epsilon_i| < \xi_i
$$
and
$$
|\kappa_j| < \zeta_j
$$
for each $i$ and $j$ in $\{0, 1, \cdots, d\}$, it is easy to check that 
\begin{align*}
|f(z) - g(z)| 
= { |f_1(z) \cdot (\sum_{j = 0}^{d} \kappa_j \cdot z^j) - f_2(z) \cdot (\sum_{i = 0}^{d} \epsilon_i \cdot z^i)| \over |f_2(z)| \cdot |g_2(z)|} < r
\end{align*}
for any $z$ in $\Omega$.

Next, for each $i$ and $j$ in $\{0, 1, \cdots, d \}$, we set the constants
$$
\xi'_i := \min \{ {s \over M'_{2} \cdot \eta^{i} }, {s \over M_{1} \cdot \eta^{i - 1}} \} > 0
$$
and
$$
\zeta'_j := \min \{ { m_{2} \over \eta^j}, {s \over M'_{1}, \eta^j }, {s \over M_{1} \cdot \eta^{j - 1} }, {s \over M_{2} \cdot \eta^j } \cdot \max\{M'_{1} \cdot M_{2}, M_{1} \cdot M'_{2} \} \} > 0
$$
Then for any rational map $g = g_1 / g_2$ in $\mathcal{R}_d$ with 
$$
g_1(z) = f_1(z) + \sum_{i = 0}^{d} \epsilon_i \cdot z^i
$$
and
$$
g_2(z) = f_2(z) + \sum_{j = 0}^{d} \kappa_j \cdot z^j
$$ 
where 
$$
|\epsilon_i| < \xi'_i
$$
and
$$
|\kappa_j| < \zeta'_j
$$
for each $i$ and $j$ in $\{0, 1, \cdots, d\}$, it is also easy to check that 
\begin{align*}
|g'_1(z) \cdot g_2(z) - g_1(z) \cdot g'_2(z)| \cdot \left| {1 \over (g_2(z))^2 } - {1 \over (f_2(z))^2 } \right|
< s
\end{align*}
and
\begin{align*}
{ 1 \over |f_2(z)|^2 } \cdot |g'_1(z) \cdot g_2(z) - g_1(z) \cdot g'_2(z) - f'_1(z) \cdot f_2(z) + f_1(z) \cdot f_2(z) | < s,
\end{align*}
thus we have 
$$
|f'(z) - g'(z)| < s
$$
for any $z$ in $\Omega$.

Therefore, by choosing a small enough neighborhood of $f$ in $\mathcal{R}_d$, we complete our proof.
\end{proof}

Now we can show the set of rational maps preserving a given subset is open in the set of rational maps if the subset is bounded and closed.

\begin{lemma}\label{4.8}
Let $\mathcal{N}_{\lambda, \Omega}$ be the set of rational maps of degree $d$ preserving a closed subset $\Omega$ and $f$ be an element in $\mathcal{N}_{\lambda, \Omega}$.
Let us assume the conditions as follows:
\begin{enumerate}
\item The rational map $f$ has no critical values and poles in $\Omega$.
\item There exists a $\delta > 0$ such that for any $z$ in $\Omega$, we have
$$
B_{z, \delta} \subset \Omega.
$$
\item There exists an $\eta > 0$ such that
$$
\Omega \subset B_{0, \eta}.
$$
\end{enumerate}
Then there exists a neighborhood $U$ of $f$ in $\mathcal{R}_d$ such that for any rational map $g$ in $U$ and any $z$ in $\Omega$, we have
$$
|f(z) - g(z) | < \mu
$$
and
$$
g \in \mathcal{N}_{\lambda, \Omega}
$$
where 
$$
\mu := \max \{ |k|^{1 / (k - 1)} \mid k \in \{ 2, 3, \cdots\} \} \cdot \delta > 0.
$$
\end{lemma}

\begin{proof}[Proof of Lemma \ref{4.8}]
By Lemma \ref{4.7}, we can choose a neighborhood $V$ of $f$ in $\mathcal{R}_d$ such that for any $g$ in $V$ and any $z$ in $\Omega$, we have
$$
|f(z) - g(z)| < \mu
$$
and 
$$
|f'(z) - g'(z)| < \lambda.
$$
We first then have
$$
|g'(z)| = |g'(z) - f'(z) + f'(z)| = \max \{ |g'(z) - f'(z)|, |f'(z)| \} = |f'(z)|
$$
for any $z$ in $\Omega$.
In particular, this implies that for any $z$ in $\Omega$, we have
$$
|g'(z)| \geq \lambda.
$$

Next, let us choose a $w$ in $\Omega$.
Then there exists a subset $\{ z_1, z_2 ,\cdots, z_d \}$ in $\Omega$ such that for any $k$ in $\{ 1, 2, \cdots, d \}$, we have
$$
f(z_k) = w.
$$
Let us fix a $k$ in $\{ 1, 2, \cdots, d\}$.
Since $f$ has no poles in $B_{z_k, \delta}$, so does $g$ by Lemma \ref{4.7}.
Thus there exists a sequence $\{ a_l \}_{l = 0}^{\infty}$ in $K$ such that for any $z$ in $B_{z_k \delta}$, we have
$$
g(z) - w 
= g(z_k) - f(z_k) + a_1 \cdot (z - z_k) + a_2 \cdot (z - z_k)^2 + \cdots
$$
and
$$
g'(z) = a_1 + 2 \cdot a_2 \cdot (z - z_k) + \cdots.
$$
Since
$$
|g'(z)| \geq \lambda > 0
$$
for any $z$ in $\Omega$,
the first derivative $g'$ of the rational map $g$ has no zeros in $B_{z_k, \delta}$.
It follows from Lemma \ref{2.9} that for any $l$ in $\{ 2, 3, \cdots \}$, we have
\begin{align*}
|a_1| \cdot {\mu \over |f'(z_k)|} &> |a_l| \cdot |l| \cdot \delta^{l-1} \cdot {\mu \over |f'(z_k)|} \\
&\geq |a_l| \cdot \mu^{l-1} \cdot {\mu \over |f'(z_k)|}
\geq |a_l| \cdot \left({\mu \over |f'(z_k)|}\right)^l
\end{align*}
On the other hand, we have
$$
|f(z_k) - g(z_k)| < \mu = |a_1| \cdot {\mu \over |g'(z_k)|} = |a_1| \cdot {\mu \over |f'(z_k)|}
$$
Thus it follows from Lemma \ref{2.9} that the restriction map 
$$
g | B_{z_k, \mu / |f'(z_k)|} \rightarrow B_{w, \mu}
$$
is bijective so there exists a unique $w_k$ in $B_{z_k, \mu / |f'(z_k)|}$ such that
$$
g(w_k) = z_k
$$
Therefore we conclude
\begin{align*}
g^{-1}(\{ w \} ) \subset \Omega.
\end{align*}
\end{proof}

Let us close this section with a discussion on repelling periodic points.

\begin{lemma}\label{4.9}
Let $\mathcal{N}_{\lambda, \Omega}$ be the set of rational maps of degree $d$ preserving a subset $\Omega$ and $f$ be an element in $\mathcal{N}_{\lambda, \Omega}$.
Let us assume the conditions as follows:
\begin{enumerate}
\item The Julia set $\mathcal{J}_f$ of $f$ is non-empty. 
\item The Julia set $\mathcal{J}_f$ of $f$ is contained in $\Omega$.
\item There exists a $\delta > 0$ such that for any $z$ in $\Omega$, we have
$$
B_{z, \delta} \subset \Omega.
$$
\end{enumerate}
Then there exists a repelling periodic point of $f$ in $\Omega_{f}$.
\end{lemma}

\begin{proof}[Proof of Lemma \ref{4.9}]
By Lemma \ref{2.8}, we have
$$
\mathcal{J}_f \subset \overline{\{ p \in K \mid \exists q \in \mathbb{N}, f^q(p) = p \}}.
$$
Thus for any $z$ in $\mathcal{J}_f$, there exists a $p$ in $B_{z, \delta}$ and a $q$ in $\mathbb{N}$ such that
$$
f^{q}(p) = p.
$$
Since 
$$
f^{-1}(\Omega) \subset \Omega
$$ 
and 
$$
f^{k}(f^{q-k}(p)) = f^{q}(p) = p,
$$
we have
$$
f^{q - k}(p) \in f^{-k}(\{ p \}) \subset \Omega
$$
for any $k$ in $\{1, 2, \cdots, q-1 \}$. 
Thus we have 
$$
|f'(w)| \geq \lambda
$$ 
for any $w$ in $\{p, f(p), \cdots, f^{q-1}(p) \}$.
Moreover, this implies that
\begin{align*}
|(f^{q})'(p)| = |f'(p)| \cdot |f'(f(p))| \cdot \cdots \cdot |f'(f^{q-1}(p))| > 1.
\end{align*}
Therefore the periodic point $p$ is repelling.
\end{proof}


\section{Proof of Theorem \ref{3.2}}

In this section, we will prove Theorem \ref{3.2}.

\begin{proof}[Proof of Theorem \ref{3.2}]
Let $\lambda$ be a real number such that
$$
|f'(z)| > \lambda
$$ 
for any $z$ in $\mathcal{J}_f$ and $\mathcal{C}_f$ be the set of the critical points, the critical values, and the poles of $f$ in $K$.

Choosing a $\delta > 0$ satisfying 
$$
\bigcup_{ z \in \mathcal{C}_f} B_{z, \delta} \cap \mathcal{J}_f = \emptyset,
$$
we have 
$$
\bigcup_{z \in \mathcal{J}_f} B_{z, \delta} \cap \mathcal{C}_f = \emptyset.
$$

Let us first shows the following lemma.

\begin{lemma}\label{5.1}
Setting 
$$
\Omega := \bigcup_{z \in \mathcal{J}_f} B_{z, \mu}
$$
where 
$$
\mu := \min \{ |k|^{1 /(k - 1)} \mid k \in \{2, 3, \cdots \} \} \cdot \delta > 0,
$$
we have the followings:
\begin{enumerate}
\item The subset $\Omega$ is bounded and closed. \\\
\item For any $z \in \Omega$, 
$$
|f'(z)| \geq \lambda.
$$
\item For any $z \in \Omega$, 
$$
f^{-1}(\{ z \} ) \subset \Omega.
$$
\end{enumerate}
\end{lemma}

\begin{proof}[Proof of Lemma \ref{5.1}]
Since $\mathcal{J}_f$ is bounded, it is clear that $\Omega$ is bounded.
It follows from the construction that $\Omega$ is also closed because $| \cdot |$ is non-Archimedean.

For any $z$ in $\Omega$, there exists a $w$ in $\mathcal{J}_f$ such that 
$$
z \in B_{w, \mu}.
$$
Since $B_{w, \mu}$ has no critical points of $f$, it follows from Proposition \ref{2.3} that
$$
|f'(z)| = |f'(w)| \geq \lambda.
$$
Moreover, by Proposition \ref{2.6}, there exists a subset $\{w_1, w_2, \cdots, w_d \}$ of $\mathcal{J}_f$ such that for any $l$ in $\{ 1, 2, \cdots, d\}$, we have
$$
f(w_l) = w.
$$
It follows from Lemma \ref{4.2} that for any distinct $l$ and $m$ in $\{ 1, 2, \cdots, d\}$, the restriction map
$$
f | B_{w_l, \mu/ |f'(w_l)| } \rightarrow B_{w, \mu}
$$ 
is homeomorphic and
$$
B_{w_l, \mu / |f'(w_l)|} \cap B_{w_m, \mu / |f'(w_m)|} = \emptyset.
$$
This implies that 
$$
f^{-1}(\{ z \}) \subset \bigcup_{n = 1}^{d} B_{w_n, \mu / |f'(w_n)|} \subset \Omega.
$$
\end{proof}

Now let us consider the set $\mathcal{N}_{\lambda, \Omega}$ of rational maps of degree $d$ preserving the subset $\Omega$.
It follows from Lemma \ref{5.1} that $\mathcal{N}_{\lambda, \Omega}$ contains $f$ so $\mathcal{N}_{\lambda, \Omega}$ is non-empty.
By Lemma \ref{4.4}, we have the sequence $\{ \Omega_{m, f} \}_{m = 0}^{\infty}$ of subsets generated by $\Omega$ and $f$ such that for any $m$ in $\{ 0, 1, \cdots \}$, we have
$$
\Omega_{\infty, f} \subset \cdots \subset \Omega_{m+1, f} \subset \Omega_{m, f} \subset \cdots \subset \Omega_{1, f} \subset \Omega_{0, f}.
$$
Since for any $m$ in $\{ 0, 1, \cdots \}$
$$
\mathcal{J}_f \subset \Omega_{m, f},
$$
we have 
$$
\mathcal{J}_f \subset \Omega_{\infty, f}.
$$
In particular, the Julia set $\mathcal{J}_f$ of $f$ is non-empty so is the limit term $\Omega_{\infty, f}$.

\begin{lemma}\label{5.2}
Let $g$ be an element in $\mathcal{N}_{\lambda, \Omega}$.
Suppose that for any $z$ in $\Omega$, we have
$$
|f(z) - g(z)| \leq \mu.
$$
Then we have the followings:
\begin{enumerate}
\item The Julia set $\mathcal{J}_g$ of $g$ is non-empty.
\item There exists a homeomorphism $h : \mathcal{J}_f \rightarrow \mathcal{J}_g$ such that for any $z$ in $\mathcal{J}_{f}$,
$$
h \circ f (z) = g \circ h (z).
$$
\end{enumerate}
\end{lemma}

\begin{proof}[Proof of Lemma \ref{5.2}]
By Lemma \ref{4.5} and Lemma \ref{4.6}, there exists a homeomorphism $h_{\infty} : \Omega_{\infty, f} \rightarrow \Omega_{\infty, g}$ such that for any $z$ in $\Omega_{\infty, f}$, we have
$$
h_{\infty} \circ f (z) = g \circ h_{\infty} (z).
$$
In particular, since $\Omega_{\infty, f}$ is non-empty, so is $\Omega_{\infty, g}$.
On the other hand, by Lemma \ref{4.9}, there exists a repelling periodic point $p$ of $f$ in the Julia set $\mathcal{J}_f$ of $f$ with period $q$.
Since $g$ is in $\mathcal{N}_{\lambda, \Omega}$ and
$$
g^{q}(h_\infty(p)) = h_{\infty}(p),
$$
we have
$$
|(g^q)'(h_{\infty}(p))| > 1.
$$
Thus the Julia set $\mathcal{J}_g$ of $g$ is non-empty by Lemma \ref{2.8}.
Furthermore, it follows from Proposition \ref{2.6} that
\begin{align*}
h_{\infty} (\mathcal{J}_f) 
&= h_{\infty} (\overline{\bigcup_{l = 0}^{\infty} {(f^{q})}^{-l}(\{ p \})}) \\
&= \overline{\bigcup_{l = 0}^{\infty} h_{\infty} ( {(f^{q})}^{-l}(\{ p \}))} \\
&= \overline{\bigcup_{l = 0}^{\infty} (g^{q})^{-l} (\{ h_{\infty} (p) \})}
= \mathcal{J}_{g}.
\end{align*}
Thus the restriction map
\begin{align*}
h : \mathcal{J}_f &\rightarrow \mathcal{J}_g, \\
z &\mapsto h_{\infty}(z),
\end{align*}
is well-defined and homeomorphic such that
$$
h \circ f (z) = g \circ h (z)
$$ 
for any $z$ in $\mathcal{J}_f$.
\end{proof}

Let us conclude that $f$ is $J$-stable in $\mathcal{R}_d$ by the following lemma.

\begin{lemma}\label{5.3}
There exists a neighborhood $U$ of $f$ such that for any $g$ in $U$, there exists a homeomorphism $h : \mathcal{J}_f \rightarrow \mathcal{J}_g$ such that for any $z$ in $\mathcal{J}_f$, we have
$$
h \circ f (z) = g \circ h (z).
$$
\end{lemma}

\begin{proof}[Proof of Lemma \ref{5.3}]
By Lemma \ref{4.8}, there exists a neighborhood $U$ of $f$ in $\mathcal{R}_d$ such that for any $g$ in $U$ and $z$ in $\Omega$, we have
$$
g \in \mathcal{N}_{\lambda, \Omega}
$$
and
$$
|f(z) - g(z)| < \mu
$$
where 
$$
\mu = \min \{ |k|^{1 /(k - 1)} \mid k \in \{2, 3, \cdots \} \} \cdot \delta > 0.
$$

Therefore by Lemma \ref{5.2}, the Julia set $\mathcal{J}_g$ of $g$ is non-empty and there exists a homeomorphism $h : \mathcal{J}_f \rightarrow \mathcal{J}_g$ such that for any $z$ in $\mathcal{J}_{f}$, we have
$$
h \circ f (z) = g \circ h (z).
$$
\end{proof}

By Lemma \ref{5.3}, the immediately expanding rational map $f$ of degree $d$ is $J$-stable in the set $\mathcal{R}_d$ of rational maps of degree $d$.
\end{proof}



\bibliographystyle{amsalpha}


\end{document}